\documentclass[11pt]{amsart}
\usepackage{amscd,amssymb}
\usepackage[arrow,matrix]{xy}
\usepackage{xcolor}
\usepackage[colorlinks=true,backref,linkcolor=blue,citecolor=purple]{hyperref}

\theoremstyle{plain}
\newtheorem{thm}[subsection]{Theorem}

\newtheorem{prop}[subsection]{Proposition}
\newtheorem{cor}[subsection]{Corollary}

\theoremstyle{definition}
\newtheorem{rk}[subsection]{Remark}
\newtheorem{definition}[subsection]{Definition}
\newtheorem{ex}[subsection]{Example}
\newtheorem{question}[subsection]{Question}

\numberwithin{equation}{section}
\newcommand{\OO}{{\mathcal O}}

\newcommand{\A}{{\mathcal A}}

\newcommand{\CC}{{\mathcal C}}

\newcommand{\Z}{\mathbb{Z}}
\newcommand{\mm}{\mathfrak{m}}
\newcommand{\Q}{\mathbb{Q}}

\newcommand{\C}{\mathbb{C}}

\newcommand{\PP}{\mathbb{P}}

\newcommand{\N}{\mathbb{N}}
\newcommand{\sat}{\mathrm{sat}}

\def\hess{\operatorname{hess}}
\DeclareMathOperator{\depth}{depth}
\DeclareMathOperator{\pd}{pd}
\DeclareMathOperator{\reg}{regularity}

\DeclareMathOperator{\indeg}{indeg}

\begin{document}

\title [The Hessian polynomial and the Jacobian ideal]
{The Hessian polynomial and the Jacobian ideal of a reduced hypersurface in $\PP^n$}

\author{Laurent Bus\'e}
\address{Universit\'e C\^ote d'Azur, Inria, 2004 route des Lucioles, 06902 Sophia Antipolis, France.}
\email{laurent.buse@inria.fr}

\author[Alexandru Dimca]{Alexandru Dimca$^{1}$}
\address{Universit\'e C\^ ote d'Azur, CNRS, LJAD, France and Simion Stoilow Institute of Mathematics,
P.O. Box 1-764, RO-014700 Bucharest, Romania}
\email{dimca@unice.fr}

\author[Hal Schenck]{Hal Schenck$^2$}
\address{Mathematics Department, Auburn University, AL 36849-5310, USA.}
\email{hks0015@auburn.edu}

\author[Gabriel Sticlaru]{Gabriel Sticlaru}
\address{Faculty of Mathematics and Informatics,
Ovidius University
Bd. Mamaia 124, 900527 Constanta,
Romania}
\email{gabriel.sticlaru@gmail.com }

\thanks{$^1$ This work has been partially supported by the Romanian Ministry of Research and Innovation, CNCS - UEFISCDI, grant PN-III-P4-ID-PCE-2020-0029, within PNCDI III. Schenck$^2$ was partially
  supported by NSF 1818646}

\subjclass[2010]{Primary 14J70, 32S05; Secondary 13D02, 32S22, 32S25}

\keywords{homogeneous polynomial, Hessian polynomial, Jacobian ideal, Castelnuovo-Mumford regularity, Milnor algebra}

\begin{abstract}  For a reduced hypersurface $V(f) \subseteq \mathbb{P}^n$ of degree $d$, the Castelnuovo-Mumford regularity of the Milnor
  algebra $M(f)$ is well understood when $V(f)$ is smooth, as well as when
  $V(f)$ has isolated singularities. We study the regularity of $M(f)$
  when $V(f)$ has a positive dimensional singular
  locus. In certain situations, we prove that the
  regularity is bounded by $(d-2)(n+1)$, which is the degree of the Hessian polynomial of
  $f$. However, this is not always the case, and we prove that in
  $\mathbb{P}^n$ the regularity of the Milnor algebra can grow
  quadratically in $d$.
\end{abstract}

\maketitle

\section{Introduction}
Let $S=\oplus_k S_k=\C[x_0,...,x_n]$ be the graded polynomial ring, where $S_k$ denotes
the vector space of degree $k$ homogeneous polynomials. For a
homogeneous polynomial $f \in S_d$, the Jacobian ideal $J_f$ is generated by
the partial derivatives of $f$, and the Milnor algebra $M(f)$ is the graded ring $S/J_f$.

From a geometric standpoint $M(f)$ is of interest because it
encodes the singular subscheme $\Sigma=\Sigma(f)$ of the projective
hypersurface $V(f) \subseteq  \PP^n$. When $V(f)$ is smooth, $M(f)$
is an Artinian complete intersection and plays a central role in the
Hodge theory of $V(f)$. A landmark result of Griffiths \cite{Griff}
shows that the Hodge numbers of $V(f)$ can be extracted from $M(f)$,
and recent work of Dimca \cite{DimcaHodge} shows that one can obtain
related results for an even dimensional nodal hypersurface. The Milnor algebra also has applications in physics, where it is known
as the Chiral ring \cite{CK}, in the study of Bernstein-Sato
polynomials (recent work of Walther \cite{Walther}), in the study of
multiplier ideals \cite{ELSV}, and in Torelli type theorems.

One case where $M(f)$ has been the object of intense investigation is when $V(f)$ is an arrangement of hypersurfaces. Of
course, in this setting $V(J_f)$ is of codimension two.
Even the simplest case, where $f$ is a product of distinct linear
forms, is highly nontrivial. A landmark result is Terao's theorem
\cite{Terao} relating the Cohen-Macaulay property of $J_f$ (often
referred to as {\em freeness}, because in this case the syzygy module
of $J_f$ is free) to the topology of the complement of the
arrangement. A second case where much is known about $M(f)$ is
when the singularities of $V(f)$ are isolated, see \cite{DP}.

Beyond the cases above, in general little is known about algebraic
properties of $M(f)$. Our aim in this paper is to investigate one of the fundamental
algebro-geometric invariants associated to $M(f)$-the {\em
  Castelnuovo-Mumford regularity}. The importance of regularity is that it measures, in a precise sense, the complexity of a finitely generated graded $S$-module $M$; it is determined by the ``shape'' of a minimal free resolution of $M$. For us, the module in question will be the Milnor algebra $M(f)$:
\begin{definition}
 Let
  \begin{equation*}
0 \to F_m \to \cdots \to F_0 \to M \to 0,
\end{equation*}
be a minimal graded free resolution of the graded $S$-module $M$, where
\[
  F_k=\bigoplus\limits_j S(-a_{k,j}) \mbox{ for }k=0, \ldots,m.
\]
By the Hilbert Syzygy Theorem, $m \leq n+1$ and $m=\pd M$ is the projective dimension of $M$. The {Castelnuovo-Mumford regularity} of $M$ is 
\[
 \reg M=\max_{i,j} \{a_{i,j}-i\}.
\]
\end{definition}
\noindent Two other central algebraic invariants are the Hilbert function and Hilbert polynomial:
\begin{definition}
For a finitely generated graded $S$-module $M$, the Hilbert function is 
$H(M,t)=\dim_{\C}M_t$, for $t \in \Z$. Hilbert proved that for $t \gg 0$, the Hilbert function is given by the Hilbert polynomial $P(M,t) \in \Q[t]$.
\end{definition}
$P(M(f),t)$ encodes information about $V(f)$; for example, if $V$ has
isolated singularities then $P(M(f),t)$ is a constant which is equal to
the sum of the Tjurina numbers at the singular points of $V(f)$. 
Castelnuovo-Mumford regularity and the Hilbert function are related via the stability threshold, defined as \[
  st(f)=\min \{q~~:~~  H(M(f),t)= P(M(f),t) \text{ for all } t \geq
  q\}.
\]
By \cite[Theorem 4.2]{Eis2}, $st(f) \leq \reg M(f) +m-n.$

When the singular subscheme $\Sigma$ is composed of finitely many
points, there is a sharp upper bound for the stability threshold
\cite{Cha}, but there are very few results when $\Sigma$ is of
positive dimension. A main motivation for our work comes from the
study of arrangements of hypersurfaces in $\mathbb{P}^n$; clearly when
$n \ge 3$ a hypersurface arrangement will have a positive dimensional
singular locus.

The square matrix of second order partial derivatives of $f$ is the
{\em Hessian} matrix, and we write $\hess_f$ to denote the determinant
of the Hessian. The polynomial $\hess_f$ has degree $T=(n+1)(d-2)$,
and it is easy to show that $\hess_f \in (J_f:\mathfrak{m})$ where
$\mathfrak{m}$ denotes the irrelevant ideal $(
x_0,\ldots,x_n)$. In fact, Spodzieja shows in \cite{Spo} (see
also Vasconcelos \cite{Vasc}) that for any hypersurface $V(f)$, then $\hess_f \in J_f \mbox{ iff } V(f) \mbox{ is singular.}$
As noted earlier, when $V(f)$ is smooth, $M(f)$ is an Artinian
complete intersection, so $M(f)_T$ is a one-dimensional $\C$-vector
space generated by the class of $\hess_f$ in $M(f)$ and $st(f)=T+1$. 
The objective of this paper is to investigate the following question:

\begin{question}\label{conj2}Find upper bounds on the $\reg M(f)$, where $V(f)$ is a
reduced singular hypersurface in $\PP^n$ of degree $d$. 
\end{question}\vskip .05in
In \S 2 and \S 3, we prove regularity bounds related to $T$. It is known that the regularity of $M(f)$ is bounded above by $T$ for
a general hypersurface, as well as for hypersurfaces that have at most
finitely many singular points. This property, which we recall with
more detail in \S \ref{sec:prelim}, suggests investigating hypersurfaces with positive dimensional singular loci. 

\medskip

\noindent {\bf Theorem:} For the following classes of reduced
hypersurfaces, $\reg M(f)< T$: \vskip .05in
\begin{enumerate}
\item $V(f)$ is a generic hyperplane arrangement in any $\PP^n$.
\item $V(f)$ is a generic (or generic symmetric) determinantal
  hypersurface in any $\PP^n$.
\item $V(f)$ is a hypersurface of degree $d \ge 3$ in $\PP^n$ which is free
  or nearly free, or a cone over a plane curve.
\item $V(f)$ is a generic arrangement of surfaces with isolated singularities in $\PP^3$.
 \end{enumerate}
The proof of all but (4) is covered in \S 2; in these cases a graded finite free resolution of $M(f)$ is available. The proof of (4) involves a delicate spectral sequence argument, and we tackle
it in \S 3. Finally, in \S 4 we show that the hope of finding an upper 
bound on regularity that is linear in $d$ is vain. We
prove: \vskip .05in
\noindent{\bf Theorem:} There exist reduced, irreducible hypersurfaces
of degree $d$ in $\mathbb{P}^n$ for which
\[
  \reg M(f) \sim \mathcal{O}(d^2).
\]
For a family of bigraded degree $d$ surfaces in $\mathbb{P}^3$, we give an explicit first syzygy of degree $\frac{d^2+d-2}{3}$; the result follows by coning over such a surface.

\section{Basic definitions and a first analysis of regularity}\label{sec:prelim}
In this section, we first recall the definitions we shall need. Then, we provide a precise statement for the upper bounds of the stability threshold and regularity of the Milnor algebra that are known for reduced hypersurfaces that are smooth or that have isolated singularities. We will end this section by  
discussing classes of reduced hypersurfaces where it is relatively easy to show that $\reg M(f) < T$ because a graded finite free resolution of the corresponding Jacobian ideal is known. 

\subsection{Basic definitions and properties}
The stability threshold and the Castelnuovo-Mumford regularity
associated to $M(f)$ were defined in the previous section. In this preliminary section, we come back briefly on these two important invariants with a particular emphasis on their connection to the local cohomology modules $H^i_\mm(M(f))$, $i\geq 0$. 

\medskip

Let $I_f$ be the saturation of the Jacobian ideal $J_f$ with respect to the ideal $\mm$ in the graded ring $S$, i.e.~$I_f=(J_f:\mm^\infty)$. The $0^\mathrm{th}$ local cohomology module measures how far $J_f$ is from being saturated; we have
 \begin{equation*}
  H^0_\mm(M(f))=I_f/J_f
 \end{equation*}
 and we denote it by $N(f)$ for simplicity. We notice that when $V(f)$ is smooth we have that $N(f)=M(f)$ so that $N(f)_T\neq 0$ and $N(f)_k=0$ for all $k>T$. 

Coming back to the regularity, one has \cite[Theorem 4.3]{Eis2}
\begin{equation*}
 \reg M(f) \geq \max \{e \ | \ N(f)_e \ne 0\},
\end{equation*}
which is in fact a consequence of the following characterization of regularity (see for instance \cite[Fact 6]{Cha}) in terms of local cohomology modules:
\begin{equation*}
 \reg M(f) =\min \{e \  | \  H^i_\mm(M(f))_{>e-i} = 0 \text{ for all non-negative integers } i \}.
\end{equation*}

\medskip

The { Hilbert function} $H(M(f)): \N \to \N$ of the graded $S$-module $M(f)$ is defined by
\begin{equation*}
 H(M(f))(k)= \dim M(f)_k,
\end{equation*}
and there is a unique polynomial $P(M(f))(t) \in \Q[t]$, called the {
  Hilbert polynomial} of $M(f)$, and an integer $k_0\in \N$ such that \cite{Eis1}
\begin{equation*}
 H(M(f))(k)= P(M(f))(k)
\end{equation*}
for all $k \geq k_0$. 
The stability threshold is defined as 
\[
  st(f)=\min \{q~~:~~  H(M(f),t)= P(M(f),t) \text{ for all } t \geq
  q\}
\]
and it is connected to local cohomology modules by means of the
Grothendieck-Serre Formula \cite[Theorem 4.3.5]{BrHe93}: for all $k\in
\mathbb{Z}$ we have
	$$ H(M(f))(k)=P(M(f))(k) + \sum_{i\geq 0} (-1)^i H(H^i_\mm(M(f)))(k).$$ 

Finally, recall that the depth of $M(f)$ is linked to the projective dimension of $M(f)$ by the Auslander-Buchsbaum formula, see \cite[Theorem A2.15]{Eis2}: 
$$\depth M(f)=\depth S-\pd M(f)=n+1-\pd M(f).$$  Moreover, \cite[Theorem A2.14]{Eis2}  tells us that
$$\depth M(f)= \inf \{k ~~|~~H^k_\mm(M(f))\ne 0 \},$$
and that $\depth M(f)=0$ if and only if $\mm$ is an associated prime of $M(f)$. In particular, $\depth M(f) >0$ if and only if $N(f)=0$.

\subsection{Isolated singularities} The behavior of the stability threshold and regularity for reduced hypersurfaces having isolated singularity have been widely studied and are well understood, in particular because the graded $S$-module $N(f)$ has a nice duality in this situation, see \cite{DS1, Se, SW}.  

Given a graded module $M$ we will denote by $\indeg(M)$ the initial degree of $M$, that is the infimum of the degrees of its nonzero elements. We have the following result.

\begin{prop}\label{prop1B} 
If the hypersurface $V(f)$ has isolated singularities then 
$$ st(f)\leq T-\indeg(N(f))+1$$
and 
$$\reg M(f) \leq T-\min\{d-1,\indeg(N(f))\}$$
unless $I_f=J_f$ is a complete intersection and $\deg(J_f):=P(M(f))=(d-1)^n$, in which case $\reg M(f) \leq T-\min\{d-2,\indeg(N(f))\}$.

In particular, $\reg M(f)< T$ except if $d=2$ and $J_f$ is a complete intersection defining a simple point, in which case $\reg M(f)=T$.
\end{prop}
\begin{proof} Under our assumption the ideal $J_f$ is an almost complete intersection of dimension 1 and the claimed inequalities follows by applying the results in \cite[Appendix]{Sz}.
\end{proof}

The above property has naturally consequences for hypersurfaces whose singular locus can be described from hypersurfaces in smaller dimension and having isolated singularities, as the following result for instance.

\begin{cor}
	For any surface in $\PP^3$ of degree $d \geq 3$ which is a cone
	over a plane curve, $\reg M(f)<T$.
\end{cor}
\begin{proof}
If $V(f)\subset \PP^3$ is a cone  over a reduced plane curve $C$, we can assume that $f$ depends only on $x_1,x_2,x_3$ and that $g=0$ is an equation for $C$ in $\PP^2$, with $g=f$. As explained in \cite[Section 3.6]{DStFS}, it follows that $M(f)=M(g) \otimes_{\C} \C[x_0]$, and hence the minimal resolution of $M(f)$ as a graded module over $S=\C[x_0,x_1,x_2,x_3]$ and of $M(g)$ as a graded module over $R=\C[x_1,x_2,x_3]$ have the same numerical invariants. This fact implies that $\reg M(f)= \reg M(g)$ and the claimed result  
follows  from Proposition \ref{prop1B}.
\end{proof}
\subsection{Some particular classes of hypersurfaces}

We begin with the classes of free  and nearly free hypersurfaces in $\PP^n$ ($n\geq 2$). A hypersurface is free if the syzygy module of $J_f$ is free \cite{DStFS}. The notion of {\em nearly free} appears for curves and surfaces in \cite{DStFScurve}, \cite{DStFS}, and was subsequently generalized by Abe in \cite{Abe}, who also defined plus-one generated arrangements. The natural generalization of this to hypersurfaces is 
\begin{definition}
A hypersurface is nearly free if its Milnor algebra $M(f)$ admits a graded free resolution of the form
\begin{multline}\label{eq:nearlyfree}
	0 \rightarrow S(-d_n-d) \rightarrow S(-d_n-(d-1))\oplus \left(\oplus_{i=1}^n S(-d_i-(d-1))\right) \\ \rightarrow S(-(d-1))^{n+1} \rightarrow S
\end{multline}   
for some integers $d_1\leq d_2\leq \ldots \leq d_n$. 
\end{definition}

\begin{prop}
\label{prop11S}
For any free or nearly free hypersurface in $\PP^n$ of degree $d \geq 3$, we have $\reg M(f)<T$.
\end{prop}
\proof
Assume first that $V(f)$ is a free hypersurface in $\PP^n$ with exponents  $(d_1,\ldots, d_n)$,  and $d_1\leq \cdots
\leq d_n$. Since the syzygy module of $J_f$ is free, the Hilbert-Burch
Theorem \cite{Eis1} shows that $\sum d_i = d-1$, and hence $d_n\leq d-1$. 
We deduce that 
$$\reg M(f) \le d-1+d_n-2 \leq 2d-4<T$$
where the last inequality holds for all $d\geq 3$. 

Now, assume that $V(f)$ is a nearly free hypersurface and that $M(f)$ admits a free resolution of the form \eqref{eq:nearlyfree} for some integers $d_1\leq \ldots\leq d_n$. The Hilbert polynomial of $M(f)$ has degree at most $n-2$ so its computation as the alternate sum of the dimensions of the graded slices of \eqref{eq:nearlyfree} in sufficiently high degree yields a condition corresponding to the vanishing of the coefficient of degree $n-1$ of the Hilbert polynomial. A straightforward computation shows that this condition is $\sum d_i=d$ (see also \cite[Proposition 4.1]{Abe}). It follows that $d_n\leq d$ and hence 
$$\reg M(f) \leq d+d_n -3 \leq 2d-3 < T$$
where the last inequality holds for all $d\geq 3$.
\endproof

We next discuss the case when $V(f)$ is a generic hyperplane arrangement $\A$, with $d=|\A| >n \geq 2$.
\begin{prop}
\label{prop12} Let $\A$ be a generic hyperplane arrangement in the projective space $\PP^n$, with $d =|\A| >n \geq 2$. Then  
$$\reg M(f)=2d-n-3 <T.$$
\end{prop}
\proof
Using the resolution for $M(f)$ when $V(f)$ is a generic hyperplane arrangement, with $d >n \geq 2$ given in \cite[Corollary 4.5.4]{RT}, 
it follows that $\depth M(f)=0$, a fact stated in  \cite[Corollary 4.5.5]{RT}, and
$\reg M(f)=2d-n-3$, since the differences $a_{i,j}-i$ are $0$ for $i=0$,
$d-2$ for $i=1$ and $2d-n-3$ for $i=2,\ldots ,n+1$. By our assumption,
$2d-n-3 \geq d-2$.
The inequality $\reg M(f) <T$ is equivalent to $(n-1) <(n-1)d$, which
clearly holds. See Ziegler \cite{Z} and Mustata-Schenck \cite{MS} for
related results. 
\endproof

A much studied class of hypersurfaces are determinantal hypersurfaces; see Beauville \cite{Beau} for results and open problems. 
We close this section with a result in this direction.
\begin{prop}\label{determinantalHS}
For an $n \times n$ generic matrix $A_n$ whose $(i,j)$ entry is
$x_{ij}$, or for a generic symmetric matrix $B_n$ whose $(i,j)=(j,i)$ entry is $x_{ij}$, we have
\[
 \reg  M(det(A_n)) = n-1 \mbox{ and } \reg  M(det(B_n)) = 2n-4.
\]
So generic and generic symmetric determinantal hypersurfaces have
$\reg M(f) < T$. 
\end{prop}
\begin{proof}
Let $a = \det(A_n)$ and $b=\det(B_n)$. Then 
\[
J_a = I_{n-1}(A_n) \mbox{  and }J_b = I_{n-1}(B_n).
\]
The resolution of the ideal of submaximal minors for a generic matrix $A_n$ is determined by Gulliksen-Neg\.ard in \cite{GN}: the ideal is Gorenstein of codimension four, and has 
\[
\reg I_{n-1}(A_n)=2n-4.
\]
In similar fashion, the ideal of submaximal minors for a generic symmetric matrix $B_n$ is determined by J\'ozefiak in \cite{Joze} (and by Lascoux in characteristic zero in \cite{L}): the ideal is Cohen-Macaulay of codimension three, and has an Eagon-Northcott resolution with
\[
\reg I_{n-1}(B_n) = n-1.
\]
Since for $J_a$ we have $T=n^2(n-2)$, and for $J_b$ we have $T=(n)(n+1)(n-2)/2$, in both cases the regularity of the Milnor ring is much smaller than $T$.
\end{proof}
\section{On generic arrangements of surfaces with isolated singularities in $\PP^3$}

In this section we show that the regularity of the Milnor algebra is bounded above by $T$ for surfaces in $\PP^3$ that are obtained as unions of surfaces with isolated singularities and in general position to some extent. These surfaces have a one-dimensional singular locus and they can be seen as a generalization of generic plane arrangements. We set $S=\C[x_0,\ldots,x_3]$ and consider $f \in S_d$ and its Jacobian ideal $J_f \subset S$ which is generated in degree $d-1$ by the four partial derivatives $f_0,\ldots,f_3$ of $f$. 

\medskip

Recall that, according to Hilbert-Burch Theorem (see \cite[Theorem 20.15]{Eis1}), a perfect graded ideal $I\subset S$ of codimension 2 admits a minimal free resolution of the form
\begin{equation}\label{eq:HB}
 0 \rightarrow \oplus_{i=1}^{r-1} S(-l_i) \xrightarrow{\Psi} \oplus_{i=1}^{r}S(-e_i) \rightarrow I \rightarrow 0	
\end{equation}
where $\Psi$ corresponds to an homogeneous matrix. In addition, we have the equality $$\sum_{i=1}^{r-1} l_i = \sum_{i=1}^r e_i=:\sigma$$ and without loss of generality, one can assume that $1\leq e_1\leq e_2 \leq \cdots \leq e_r$ and $l_1\leq l_2 \leq \cdots \leq l_{r-1}$. We notice that the minimality assumption implies that $e_1<l_1$.
We begin with the following rather general result.

\begin{thm}\label{thm:ques16} Suppose $I=(g_1,\ldots,g_r)$ be a perfect ideal in $S$ of codimension 2 with a minimal free resolution of the form \eqref{eq:HB}, and $f \in S$ is such that $J_f \subset I$ and that the ideal $(J_f:I)$ defines a 0-dimensional subscheme in $\PP^3$ (possibly empty). Then
	$$st(f)\leq \max\{ 4d-7-2e_1, l_1-3 \}$$ 
	and
	$$\reg M(f) \leq 
	\begin{cases}
	\max\{ 4d-8-2e_1 , l_1-2 \}  \textrm{ if } e_1<d-1, \\
	\max\{ 2d-5 , l_1-2 \}  \textrm{ if } e_1=d-1.		
	\end{cases}$$
Moreover, in these conditions, $st(f)\leq T$, and $\reg M(f) <T$ providing $l_1\leq T+1$.
\end{thm}

\begin{rk} Since $\sum_{i=1}^{r-1} l_i = \sum_{i=1}^r e_i$ we have $(r-1)l_1\leq r e_r$. Therefore the condition $l_1\leq T+1$ is satisfied if
	$e_r\leq \frac{r-1}{r}(T+1)$.
\end{rk}

\begin{proof} By the Grothendieck-Serre Formula we have that for all $k\in \mathbb{Z}$ one has
	$$ H(M(f))(k)=P(M(f))(k) + \sum_{i\geq 0} (-1)^i H(H^i_\mm(M(f)))(k).$$ 
Moreover, since $\dim(M(f))\leq 2$ we have $H^i_\mm(M(f))=0$ for all $i>2$ by \cite[Theorem 3.5.7]{BrHe93}. Therefore, to prove the claimed result we need to examine the vanishing of the graded components of  the local cohomology modules $H^i_\mm(M(f))$ for $i=0,1,2$.

\medskip 

The hypotheses of Theorem~\ref{thm:ques16} implies that the following graded complex is a minimal free resolution of $I$:
$$ F_\bullet \ : \ F_2:=\oplus_{i=1}^{r-1}S(-l_i) \xrightarrow{\Psi} F_1:=\oplus_{i=1}^r S(-e_i) \xrightarrow{(g_1\ \ldots \ g_r)} S.$$ 
By examining the two spectral sequences corresponding to the filtrations by rows and columns of the double complex $\CC_\mm^\bullet(F_\bullet)$, we deduce immediately that $H^0_\mm(S/I)=H^1_\mm(S/I)=0$ and the graded isomorphism 
$$H^2_\mm(S/I)\simeq \ker (H^4_\mm(F_2) \rightarrow H^4(F_1))$$
where the map $H^4_\mm(F_2) \rightarrow H^4(F_1)$
is the canonical one induced by the Koszul complex. In particular, we have 
$H^2_\mm(S/I)_k=0$ 
for all $k>l_1-4$.

\medskip
	
Now, consider the canonical exact sequence
\begin{equation}\label{eq:ses}
 0 \rightarrow I/J_f \rightarrow M(f)=S/J_f \rightarrow S/I \rightarrow 0.
\end{equation}
We know that $H^i_\mm(S/I)=0$ for $i=0,1$. We also have $H^i_\mm(I/J_f)=0$ for all $i>1$ because $\mathrm{ann}_S(I/J_f)=(J_f:I)$ and $(J_f:I)$ is assumed to define a $0$-dimensional subscheme in $\PP^3$. Therefore, the long exact sequence of local cohomology of \eqref{eq:ses} implies that 
$$H^0_\mm(M(f))\simeq H^0_\mm(I/J_f), \ \ H^1_\mm(M(f))\simeq H^1_\mm(I/J_f)$$ 
$$\textrm{ and } H^2_\mm(M(f))\simeq H^2_\mm(S/I).$$
An immediate consequence is that 
\begin{equation}\label{eq:H2QMf}
H^2_\mm(M(f))_k=0 \textrm{ for all } k>l_1-4=:\eta_2.	
\end{equation}
To examine $H^i_\mm(I/J_f)$, $i=0,1$, we proceed as follows.

From the inclusion $J_f\subset I$, one can decompose the $f_i$'s on $g$ and $g'$ to get a $r\times 4$-matrix $H$ such that
	$$ \left( f_0 \ f_1 \ f_2 \ f_3 \right) 
	= \left( g_1 \ \cdots \ g_r \right) H
	= \left( g_1 \ \cdots \ g_r \right) 
	\left(
	\begin{array}{cccc}
	h_{0,1} & h_{1,1} & h_{2,1} & h_{3,1} \\
	h_{0,2} & h_{1,2} & h_{2,2} & h_{3,2} \\
	\vdots & \vdots & \vdots & \vdots \\
	h_{0,r} & h_{1,r} & h_{2,r} & h_{3,r} \\
	\end{array}
	\right).$$ 
	This latter corresponds to an homogeneous map 
	$$K_1=S(-(d-1))^4 \xrightarrow{H} F_1$$
	that gives rise to a finite free graded presentation of the quotient $I/J_f$, namely the graded exact sequence
	$$ K_1'=F_2 \oplus K_1 \xrightarrow{\varphi} F_1  
	\xrightarrow{(g_1,\ldots,g_r)} I/J_f \rightarrow 0$$
	where the map $\varphi:K_1'\rightarrow F_1$ is defined by the $r\times(r+3)$ matrix
	   $$\left(\begin{array}{c|cccc}
	    & h_{0,1} & h_{1,1} & h_{2,1} & h_{3,1} \\
	   \Psi &\vdots  & \vdots & \vdots & \vdots \\ 
	    & h_{0,r} & h_{1,r} & h_{2,r} & h_{3,r} \\
	   \end{array}\right).$$ 
	  The Buchsbaum-Rim complex $E_\bullet$ associated to $\varphi$, that belongs to the family of generalized Koszul complexes \cite[Appendix A.2.6]{Eis1}, is the following graded complex 
	\begin{multline*}
	 E_4= S_2(F_1^*)\otimes \wedge^{r+3}(K_1')(\sigma) \rightarrow E_3= S_1(F_1^*)\otimes \wedge^{r+2}(K_1')(\sigma) \rightarrow \\
E_2= S_0(F_1^*)\otimes \wedge^{r+1}(K_1')(\sigma) \rightarrow E_1=K_1' \xrightarrow{\varphi} E_0=F_1
	\end{multline*}
	where $F_1^*$ denotes the dual of $F_1$ and $\sigma=\sum_{i=1}^r e_i=\sum_{i=1}^{r-1}l_i$.
	It is a classical property that the homology of $E_\bullet$ is supported on 
	$$\mathrm{ann}_S(\mathrm{coker}(\varphi))=\mathrm{ann}_S(I/J_f)=(J_f:_S I)$$ 
	and by the hypotheses of Theorem~\ref{thm:ques16} this is a finite subscheme in $\PP^3$. 
	 
	 Now, consider the double complex $\CC_\mm^\bullet(E_\bullet)$. The spectral sequence corresponding to its filtration by rows converges at the second step and is of the form 
 	{\small 
	\[
 	\xymatrixrowsep{0.4em}
	\xymatrixcolsep{.5em}
 	\xymatrix{
 	H^0_\mm({H_{4}}(E_\bullet)) & H^0_\mm({H_{3}}(E_\bullet)) & H^0_\mm({H_{2}}(E_\bullet))&H^0_\mm({H_{1}}(E_\bullet)) &H^0_\mm(I/J_f)  \\
	H^1_\mm({H_{4}}(E_\bullet)) & H^1_\mm({H_{3}}(E_\bullet)) & H^1_\mm({H_{2}}(E_\bullet))&H^1_\mm({H_{1}}(E_\bullet)) &H^1_\mm(I/J_f)  \\
 	{0} & {0} & 0& 0 &0\\
 	\vdots & \vdots & \vdots & \vdots & \vdots \\
 	}
 	\]}
	
	\noindent (observe that $H_0(E_\bullet)=I/J_f$). The spectral sequence corresponding to the filtration by columns of $\CC_\mm^\bullet(E_\bullet)$ also converges at the second step with a single non-zero row: $H_\bullet(H^4_\mm(E_\bullet))$.  Comparing these two spectral sequences, we deduce that  
	\begin{itemize}
		\item $H^0_\mm(I/J_f)_{k}=0$ for all  $k$ such that $H^4_\mm(E_4)_{k}=0$ and
		\item $H^1_\mm(I/J_f)_{k}=0$ for all  $k$ such that $H^4_\mm(E_3)_{k}=0$.
	\end{itemize}
But from the description of $E_\bullet$ we have 
	\begin{align*}
		E_4 &= S_2(F_1^*)\otimes \wedge^{r+3}(K_1')(\sigma) \\
		    &\simeq \oplus_{1\leq i\leq j\leq r} S(-4(d-1)+e_i+e_j)
	\end{align*}
from we deduce that $H^4_\mm(E_4)_{k}=0$, hence $H^0_\mm(M(f))_k\simeq H^0_\mm(I/J_f)_{k}=0$, for all integers
\begin{equation}\label{eq:H0QMf}
k>4(d-1)-4-2e_1=4d-8-2e_1=:\eta_0,	
\end{equation}
We also have 
	\begin{align*}
		E_3 &= S_1(F_1^*)\otimes \wedge^{r+2}(K_1')(e+e') \\
		    &\simeq \oplus_{i=1}^r S(-3(d-1)+e_i)^4 \oplus_{i=1}^r\oplus_{j=1}^{r-1} S(-4(d-1)+l_j+e_i)
	\end{align*}
from we deduce that $H^4_\mm(E_3)_{k}=0$, hence $H^1_\mm(M(f))_k\simeq H^1_\mm(I/J_f)_{k}=0$, for all integer
\begin{equation}\label{eq:H1QMf}
 k > 3d-7-e_1+(d-1-l_1)_{+}=:\eta_1,
\end{equation} 
where $(d-1-l_1)_{+}=\max\{0,d-1-l_1\}$.

From here, the claimed results follows by comparing the constraints given by \eqref{eq:H0QMf}, \eqref{eq:H1QMf} and \eqref{eq:H2QMf}. Indeed, we have 
$$\eta_0-\eta_1=(d-1-e_1) - (d-1-l_1)_+
=
\begin{cases}
	l_1-e_1 & \textrm{ if }  l_1\leq d-1, \\
	d-1-e_1 & \textrm{ if } l_1 \geq d-1.
\end{cases}
$$
Since $e_1\leq d-1$ ($J_f\subset I$), and $l_1>e_1$ (for otherwise the first column of $\Psi$ would be identically zero), the quantity $\eta_0-\eta_1$ is always non-negative. More precisely, $\eta_0>\eta_1$ if $1\leq e_1< d-1$ and $\eta_0=\eta_1=2d-6$ if $e_1=d-1$.

Gathering all the previous results, we deduce the following properties. First, $N(f)_k=0$ for all $k\geq T$ because $\eta_0<T$, since $e_1\geq 1$ ($I$ is of codimension 2). Second, we have
$$st(V)\leq \max\{\eta_0,\eta_1,\eta_2\}+1=\max\{\eta_0+1,\eta_2+1\}.$$
Finally, since $\reg M(f)=\max\{\eta_0,\eta_1+1,\eta_2+2\}$ the claimed inequalities follows from the two cases $e_1<d-1$, for which $\max\{\eta_0,\eta_1+1\}=\eta_0=T-2e_1$, and $e_1=d-1$, for which $\max\{\eta_0,\eta_1+1\}=\eta_0+1=2d-5$. Therefore $\reg(M(f))<T$, if $d\geq 2$ and $\eta_2+2<T$, this latter condition being equivalent to $l_1\leq T+1$.
\end{proof}

Now we introduce the generic arrangements of surfaces with isolated singularities in $\PP^3$ which appear in the title of this section.
\begin{thm} 
\label{thmSA}

Suppose given a collection of $r$ surfaces $D_i=V(f_i)$, $i=1,\ldots,r$, in $\PP^3$ of positive degree $1\leq d_1\leq \ldots\leq d_r$ respectively and consider the surface $V=\cup_{i=1}^r D_i=V(f)$, where $f :=\prod_{i=1}^r f_i$, in $\PP^3$ of degree $d=\sum_{i=1}^r d_i \geq 2$. Assume that
\begin{itemize}
	\item  $f$ is a reduced polynomial,
	\item  $D_i$ has only finitely many singular points for all $i$,
	\item  the intersection between any two distinct surfaces $D_i$ and $D_j$ is transverse, except at finitely many points, 
	\item  the intersections between any three distinct surfaces $D_i$ consists in finitely many points.
\end{itemize}	
Then
\begin{equation}\label{eq:stfi}
	st(f)\leq 2d+2d_r-7	
\end{equation}
	and hence $st(f)\leq T$ providing $d\geq 3$.
		
	Moreover, in the generic setting, more precisely if the surfaces $D_i$, $i=1,\ldots,r$, are all smooth surfaces that intersect transversally at all their intersection points (in particular any four distinct surfaces do not intersect), 
then we have $I_f=(g_1,\ldots,g_r)$, where $g_i:=f/f_i$ for all $i=1,\ldots,r$, and \eqref{eq:stfi} is an equality. 
\end{thm}

\begin{proof}
Let $I$ be the ideal of $S$ generated by $g_1,\ldots,g_r$. A straightforward computation shows that $J_f \subset I$. Actually, we have the equality
$$
\left( \partial_0 f \ \partial_1 f \ \partial_2 f \ \partial_3 f \right)=
\left( g_1 \ g_2 \ \cdots \ g_r \right) \cdot
\left(
\begin{array}{cccc}
	\partial_0 f_1 & \partial_1 f_1 & \partial_2 f_1 & \partial_3 f_1 \\
	\partial_0 f_2 & \partial_1 f_2 & \partial_2 f_2 & \partial_3 f_2 \\
	\vdots & \vdots & \vdots & \vdots \\
	\partial_0 f_r & \partial_1 f_r & \partial_2 f_r & \partial_3 f_r 
\end{array}
\right)
$$
where the matrix on the right is the Jacobian matrix of the $f_i$'s; we denote it $H$. It defines a graded map 
$$ R(-d+1)^4 \xrightarrow{H} \oplus_{i=1}^r R(-d+d_i).$$ 
On the other hand, consider the following matrix: 
$$ 
\Psi=\left( 
\begin{array}{cccccc}
f_1 & 0 &  &  & \cdots & 0 \\
-f_2 & f_2 & 0 & & \cdots & \\ 
0 & - f_3 & f_3 & \ddots & \cdots & \vdots  \\	
\vdots &  0 & \ddots & \ddots & 0 &  \\
\vdots &  0 & \ddots & \ddots & f_{r-2} & 0 \\
 & \ldots &  & 0 & -f_{r-1} & f_{r-1} \\	
0 & \ldots &  &  & 0 & -f_r 
\end{array}
\right).
$$
It is of size $r\times (r-1)$ and its $(r-1)$-minors coincide with the $g_i$'s up to sign. Therefore, assuming that $I$ has codimension 2, i.e.~that $f$ is a reduced polynomial, then Hilbert-Burch Theorem implies that $I$ admits the following minimal free resolution 
\begin{equation}\label{eq:HBr}
 0 \rightarrow \oplus_{i=1}^{r-1} S(-d) \xrightarrow{\Psi} \oplus_{i=1}^{r}S(-d+d_i) \rightarrow S \rightarrow S/I \rightarrow 0	
\end{equation}

Now, as already used in the proof of Theorem \ref{thm:ques16}, the concatenation of the matrices $\Psi$ and $H$ provides a free presentation of $I/J_f$. Therefore, we deduce that the ideal $I_r(\Psi\oplus H)$ of $r-$minors of this concatenated matrix has the same radical as the ideal $(J_f:I)$ (using a classical property of Fitting ideals \cite[Proposition 20.7]{Eis1}). In addition, examining the matrix $\Psi \oplus H$ we notice that
\begin{itemize}
	\item If $p \in D_i \setminus \cup_{j\neq i} D_j$ and $p \notin \mathrm{Sing}(D_i)$, then $p \notin V(J_f:I)$. This is because the partials $\partial_i (f)$ are obtained as $r$-minors of $\Psi \oplus H$ (take $\Psi$ and add a column of $H$).
	\item If for all points $p \in D_i \cap D_j$, except finitely many, the intersection of $D_i$ and $D_j$ at $p$ is transverse and $p$ is not contained in any other surfaces $D_l$, then there exists an $r$-minor of $\Psi \oplus H$ that does not vanish: take $\Psi$, remove number $i$ and $j$ and replace them by those 2 columns of $H$ corresponding to $\partial_k$ and $\partial_l$ such that the Jacobian minor $\partial_k f_i \partial_l f_j - \partial_k f_j  \partial_l f_i$ is nonzero.
\end{itemize} 
Under our assumptions, we deduce that $(J_f:I)$ is supported on finitely many points and hence one can apply Theorem \ref{thm:ques16} with the data: $e_1=d-d_r$ and $l_1=d$. The conclusion \eqref{eq:stfi} follows because we always have $2d+2d_r-7\geq d-3$ as $d\geq 2$ and $d_r\geq 1$. 

Now, we turn to the proof of the second part of this theorem. Pushing further the above analysis of $r$-minors of the matrix $\Psi\oplus H$, one can show in the same way that for all point $p\in V$ there exists an $r$-minor of the matrix $\Psi\oplus H$ that does not vanish at $p$, under our genericity assumptions. Therefore, the ideal $(J_f:I)$ defines an empty algebraic variety, which means that the ideals $J_f$ and $I$ have the same saturation, so that $J_f^\sat =I_f =I^\sat =I$. Thus, taking again the proof of Theorem \ref{thm:ques16}, this latter property implies that 
$$H^0_\mm(I/J_f)\simeq \ker \left( H^4_\mm(E_4) \rightarrow H^4_\mm(E_3)\right)$$
because $H^1_\mm(H_i(E_\bullet))=0$ for all $i\geq 0$. If $\eta_0>\eta_1$, then $H^4_\mm(E_3)_{\eta_0}=0$ whereas $H^4_\mm(E_4)_{\eta_0}\neq 0$ so we deduce that $H^0_\mm(I/J_f)_{\eta_0} \neq 0$. As we already observed, the condition  $\eta_0>\eta_1$ holds if and only if $e_1<d-1$, and since $e_1=d-d_r$, it holds if and only if $d_r\geq 2$. It remains to consider the case $d_1=\cdots=d_r=1$, for which we have $l_1=\ldots=l_{r-1}=d=r$ and $\eta_0=\eta_1=2d-6$.  Again, from the proof of Theorem \ref{thm:ques16} we have 
$$E_4\simeq S(-2(d-1))^{\binom{r+1}{2}} \ \mathrm{ and } \ E_3\simeq S(-2(d-1))^4 \oplus S(-2(d-1)+1)^{r(r-1)}.$$
It follows that 
$$H^4_\mm(E_4)_{\eta_0}\simeq H^4_\mm(S)_{-4}^{\binom{r+1}{2}}\simeq \C^{\binom{r+2}{2}}$$
and 
$$H^4_\mm(E_3)_{\eta_0}\simeq H^4_\mm(S)_{-4}^{4}\oplus H^4_\mm(S)_{-3}^{r(r-1)}\simeq \C^4.$$
Therefore, if $\binom{r+1}{2}>4$, i.e.~if $r>2$, then necessarily $H^0_\mm(I/J_f)_{\eta_0} \neq 0$. In the end, the theorem is proved except for the case where $f$ is the product of two linearly independent planes ($d=2$, $r=2$, $d_1=d_2=2$). But this case can be treated directly from the definitions: $M(f)$ is isomorphic to a graded polynomial ring in two variables, hence $H(M(f))(k)=\max\{k+1,0\}$ and $P(M(f))(k)=k+1$ for all $k\in \mathbb{Z}$, so that $st(f)=-1=2\cdot2+2\cdot1-7$.
\end{proof}

\begin{ex}
\label{ex1}
In this example we show that the inequality for $st(f)$ in Theorem \ref{thmSA} can be either strict, or an equality, even for $r=2$ surfaces.
First we consider the family of surfaces $V_{d+1}=V(f_1f_2)$, where
$f_1=x_3$ and $f_2=x_0^{d}+x_1^{d}+x_2^{d}$ with $d\geq 2$.
Then $D_1=V(f_1)$ is a plane, $D_2=V(f_2)$ is a surface with a singular point at $p=(0:0:0:1)$ with local Tjurina number $\tau(D_2,p)=(d-1)^3$, and the intersection $C=D_1\cap D_2$ is transverse.
Hence Theorem \ref{thmSA} applies and gives the inequality
$$st(f) \leq 2(d+1)+2d-7=4d-5.$$
Using that the Hilbert polynomial $P(M(f))$ is a linear function of the form $ak+b$ 
and the values for $a,b$ given in \cite[Formula (2.6) and subsection (3.1)]{DStFS}, we see that
$$P(M(f))(k)=ak+b=\deg(C) k+ \chi(C,\OO_C)+ \tau_0(V)=dk-\frac{d(d-3)}{2}+(d-1)^3,$$
since in this case $\tau_0(V)=\tau(D_2,p)$.
A direct computation of the Hilbert series for $V_{d+1}$ when $2\leq d\leq 5$ using SINGULAR, shows that in all these cases
$$st(f)=3d-5<4d-5.$$
Next we consider the family of surfaces $V'_{2d}=V(f_1f_2)$, where
$f_1=x_1^{d}+2x_2^{d}+x_3^{d}$ and $f_2=x_0^{d}+x_1^{d}+x_2^{d}$ with $d\geq 2$.
Then $D_1=V(f_1)$ is a surface with a singular point at $q=(1:0:0:0)$ with local Tjurina number $\tau(D_1,q)=(d-1)^3$, the surface $D_2$ is as above, and the intersection $C=D_1\cap D_2$ is transverse.
Hence Theorem \ref{thmSA} applies again and gives the inequality
$$st(f) \leq 2(2d)+2d-7=6d-7.$$
As above, for the Hilbert polynomial, we get 
$$P(M(f))(k)=d^2k+d^3-4d^2+6d-2,$$
since in this case $\tau_0(V)=\tau_q(D_1,q)+\tau_p(D_2,p)$.
A direct computation of the Hilbert series for $V'_{2d}$ when $2\leq d\leq 5$ shows that in all these cases $st(f)=6d-7.$
\end{ex}
\begin{rk}
\label{rkex1}
By Theorem \ref{thmSA} when $r=2$ then $st(f)=2d+2d_2-7,$ the surfaces $D_1$ and $D_2$ are smooth, and the intersection $C=D_1\cap D_2$ is transverse; 
see \cite[Question 3.2]{DStFS}.
\end{rk}

\section{Regularity of $M(f)$ for hypersurfaces in $\PP^n$ singular in codimension one}\label{sec:quadreg}
The main result in this section is in contrast to the results in
the previous sections: we prove that a reduced hypersurface $V(f)
\subseteq \PP^n$ of degree $d$ which has a non-reduced codimension one singular component can have 
\[
\reg M(f_d) \sim \mathcal{O}(d^2). 
\]
We prove the result for $n=3$; the general case then follows by coning. In fact, Theorem~\ref{Codim1SurfReg} below yields a stronger result,
giving an explicit minimal first syzygy of high degree. The surfaces which we analyze are bigraded, so have a singular locus containing two disjoint lines. We utilize the extra structure provided by the bigrading in an essential way. 
\begin{thm}\label{Codim1SurfReg}
Let $f_{(k,d-k)} \in \C[x_0,\ldots,x_3]$ be a general bihomogeneous polynomial of bidegree $(k,d-k)$ with $k, d-k \ge 1$; note that $f \in (x_0,x_1)^k \cap (x_2,x_3)^{d-k}$. Let $D=V(f_{(k,d-k)})$ be the corresponding surface in $\PP^3$, which is of degree $d$ in the $\Z$-grading. Then $D$ is reduced, and if $d-k$ is odd, $M(f_{(k,d-k)})$ has a minimal first syzygy of bidegree 
\[(2k,0) + \frac{d-k-1}{2}(3k-2,3).\]
Hence, in the $\mathbb{Z}$-grading 
\[
\reg{M(f_{(k,d-k)})} \ge \frac{3k+1}{2}d -\frac{3k^2+5}{2}
\]
and taking $k=\frac{d-1}{3}$ yields
\[
\reg{M(f_{(\frac{d-1}{3}, \frac{2d+1}{3})})} \ge \frac{1}{3}(d^2+d-8). 
\]
\end{thm}

\begin{proof} The fact that $D$ is reduced follows from the hypothesis that $f$ is general. Define the (standard) graded polynomial ring $A=\C[x_0,x_1]$, so that the polynomial ring $S=\C[x_0,\ldots,x_3]=A[x_2,x_3]$ inherits a bigraded structure. In particular, $f_{(k,d-k)} \in S_{k,d-k}$. Denoting by $f_i$ the partial derivative of $f_{(k,d-k)}$ with respect to $x_i$, we have $f_0,f_1 \in S_{k-1,d-k}$ and $f_2,f_3 \in S_{k,d-k-1}$. In other words, the Jacobian ideal $J_{f_{(k,d-k)}}=J=(f_0,\ldots,f_3)$ has a bigraded presentation
	\begin{equation}\label{eq:presJ}
		S(-k+1,-d+k)^2\oplus S(-k,-d+k+1)^2 \xrightarrow{\phi} S \rightarrow S/J \rightarrow 0.
	\end{equation}
	
Since $x_0f_0+x_1f_1=kf_{(k,d-k)}$ and $x_2f_2+x_3f_3=(d-k)f_{(k,d-k)}$, the ideal $J$ admits the bi-Euler syzygy $((d-k)x_0,(d-k)x_1,kx_2,kx_3)$. This syzygy is of bidegree $(k,d-k)$ in the first syzygy module $\mathrm{Syz}$ of $J$, that is the kernel of $\phi$.

\medskip

Henceforth, we use {\em minimal} to mean a syzygy of the smallest possible degree with respect to the variables $\{x_2,x_3\}$, which is not the bi-Euler syzygy. Our goal is to find a minimal syzygy of $J$; to do this we consider the graded slice of \eqref{eq:presJ} in degree $\eta \in \N$ with respect to $\{x_2,x_3\}$:
$$  0 \rightarrow \mathrm{Syz}_{*,\eta} \rightarrow A(-k+1)^{2(\eta-d+k+1)}
\oplus A(-k)^{2(\eta-d+k+2)} \xrightarrow {\phi_\eta} A^{\eta+1}.$$
By the Hilbert Syzygy Theorem, $\mathrm{Syz}_{*,\eta}$ is a free $A$-module. Moreover, taking into account the bi-Euler syzygy we have 
$$ \mathrm{Syz}_{*,\eta} \simeq M_\eta \oplus A(-k)^{\eta-d+k+1}.$$

In order to have control on the degree of a minimal syzygy, we seek the smallest value of $\eta$ such that $M_\eta$ is nonzero. If in addition $M_\eta$ has rank one, then the minimal syzygy will be determinantal. If the maps $\phi_{\eta}$ have expected ranks for small values of $\eta$ then will occur when 
$$2(\eta-d+k+1)+2(\eta-d+k+2)-(\eta-d+k+1)=\eta+2,$$
or equivalently $2\eta=3(d-k-1)$. So, if $d-k$ is odd, i.e.~if $d=k+2\mu+1$ for some integer $\mu\in \N$, then a minimal syzygy is expected in degree $(e,\eta=3\mu)$, $e\in \N$. 
We claim that the maps $\phi_\eta$ have expected ranks for all $\eta\leq 3\mu$:
\begin{equation}\label{eq:rankphi}
	M_{\eta}=0 \mathrm{ \ for \ all \ } \eta<3\mu \mathrm{ \ and \ } M_{3\mu}\simeq A(-e),
\end{equation}
which implies that this syzygy of degree $(e,3\mu)$ is a minimal syzygy of $J$. To prove the claim we focus on the particular structure of the maps $\phi_\eta$. 

\medskip

In standard monomial bases, the matrix of $\phi_\eta$ is built from Sylvester blocks:
\begin{equation}\label{eq:4blockSylv}
\left( 
\begin{array}{c|c|c|c}
 \vdots & \vdots & \vdots &  \vdots \\
\mathrm{Sylv}_{\eta}(f_0) & \mathrm{Sylv}_{\eta}(f_1) & \mathrm{Sylv}_{\eta}(f_2) & \mathrm{Sylv}_{\eta}(f_3) \\
 \vdots & \vdots & \vdots &  \vdots 
\end{array}
\right).
\end{equation}
The polynomials $f_i$ are seen as homogeneous polynomials in $x_2,x_3$. The blocks  
$\mathrm{Sylv}_{\eta}(f_0)$ and $\mathrm{Sylv}_{\eta}(f_1)$ are of size $(\eta+1)\times (\eta-2\mu)$ since $f_0$ and $f_1$ are of degree $2\mu+1$ in $x_2,x_3$, and the blocks $\mathrm{Sylv}_{\eta}(f_2)$ and $\mathrm{Sylv}_{\eta}(f_3)$ are of size $(\eta+1)\times (\eta-2\mu+1)$ since $f_2$ and $f_3$ are of degree $2\mu$ in $x_2,x_3$. In view of these dimensions, it is clear that $M_{\eta}=0$ for all $\eta<2\mu$. Also, the matrix of $\phi_{2\mu}$ has two columns, namely the coefficients of $f_2$ and $f_3$ in degree $\eta=2\mu$; $\phi_{2\mu}$ is clearly injective and hence $M_{2\mu}=0$.

For $\eta\geq 2\mu+1$ the bi-Euler syzygy appears and yields
$\eta-2\mu$ independent syzygies in degree $\eta$ that correspond to
the following relations among the columns of the matrix of
$\phi_\eta$, for all $i=0,\ldots,\eta-2\mu-1$: 
\begin{align*}
	0 &= x_2^{\eta-2\mu-1-i}x_3^{i}\left( (d-k)x_0f_0+(d-k)x_1f_1-kx_2f_2-kx_3f_3 \right) \\
	&= (d-k)x_0(x_2^{\eta-2\mu-1-i}x_3^{i}f_0)+(d-k)x_1(x_2^{\eta-2\mu-1-i}x_3^{i}f_1)  \\
	&\hspace{1em}   - k(x_2^{\eta-2\mu-i}x_3^{i}f_2) - k(x_2^{\eta-2\mu-1-i}x_3^{i+1}f_3)
\end{align*}
Removing these relations from the matrix of $\phi_\eta$, we obtain a new map whose matrix is 
\begin{equation}\label{eq:fullSylv}
\left( 
\begin{array}{c|c|c|c}
 \vdots & \vdots & \vdots &  \vdots \\
\mathrm{Sylv}_{\eta}(f_0) & \mathrm{Sylv}_{\eta}(f_1) & \mathrm{Sylv}_{\eta}(f_2) & x_2^{\eta-2\mu}f_3 \\
 \vdots & \vdots & \vdots &  \vdots 
\end{array}
\right).
\end{equation}
It is obtained by removing the last $\eta-2\mu$ columns to  \eqref{eq:4blockSylv} (the block on the right is a single column). The matrix of \eqref{eq:fullSylv} is of size $(\eta+1)\times(3\eta-6\mu+2)$ and its kernel is $M_\eta$. When $\eta< 3\mu$ it has less columns than rows and to prove the claim we must show that those columns are $A$-independent, or equivalently by McCoy's Lemma, that there exists a nonzero minor of size $3\eta-6\mu+2$. To achieve this, we examine in more detail the Sylvester blocks. We set (recall $d-k=2\mu+1$)
\begin{equation*}
f_{(k,d-k)}=x_2^{2\mu+1}h_{0}(x_0,x_1)+x_2^{2\mu}x_3h_1(x_0,x_1)+\ldots
+x_3^{2\mu+1}h_{2\mu+1}(x_0,x_1)
\end{equation*}
where $h_i(x_0,x_1)$ are homogeneous polynomials of degree $k$ in $A$.
With this notation, the Sylvester blocks can be written more explicitly; for instance we have
$$
\mathrm{Sylv}_\eta(f_0)=\left( 
\begin{array}{ccccc}
\partial_{x_0}h_0 & 0  & \cdots & 0 \\
\partial_{x_0}h_1 & \partial_{x_0}h_0  & \cdots & 0 \\	
\vdots & \vdots &  \ddots &  0 \\
0 &  \cdots & 0 & \partial_{x_0}h_{2\mu+1} \\	
\end{array}
\right)$$
which is of size $(\eta+1)\times(\eta-2\mu)$ and whose rows are indexed from top to bottom by the monomial basis $\{x_2^\eta,x_2^{\eta-1}x_3,\ldots,x_3^\eta \}$. Thus the top maximal minor of this matrix is a lower triangular matrix of determinant $\left(\partial_{x_0}h_0\right)^{\eta-2\mu}$. Similarly, the bottom minor of $\mathrm{Sylv}_\eta(f_1)$ is a lower triangular matrix of determinant $\left(\partial_{x_1}h_{2\mu+1}\right)^{\eta-2\mu}$. Removing the rows of these two minors from \eqref{eq:fullSylv} we are left with the following submatrix of size 
\[(4\mu-\eta+1)\times(\eta-2\mu+2)\]
whose rows are indexed from top to bottom by the monomial basis $\{x_2^{2\mu}x_3^{\eta-2\mu}, \ldots, x_2^{\eta-2\mu}x_3^{2\mu}\}$:
$$
\left(
\begin{array}{ccc|c}
 (4\mu-\eta) h_{\eta-2\mu}  &  \ldots & (2\mu+1)h_0 & (\eta-2\mu+1)h_{\eta-2\mu+1} \\

 (4\mu-\eta-1) h_{\eta-2\mu+1}  & \ldots & (2\mu)h_1 & (\eta-2\mu)h_{\eta-2\mu+2} \\
 
\vdots &   & \vdots &  \ldots \\

h_{2\mu} & \ldots &  (\eta-2\mu+1)h_{4\mu-\eta} & (2\mu+1)h_{2\mu+1} \\	
\end{array}\right).
$$
From here, one can choose the $(\eta-2\mu+2)$-minor built with the top minor of maximal size $(\eta-2\mu+1)$ in the left block, whose determinant is denoted by $\Delta$, and the bottom minor of size 1 in the right block, whose determinant is obviously equal to $(2\mu+1)h_{2\mu+1}$. In the end, we have selected a $(3\eta-6\mu+2)$-minor in \eqref{eq:fullSylv} whose determinant is equal to 
\begin{equation}\label{eq:det}
(2\mu+1)h_{2\mu+1} 
\left(\partial_{x_0}h_0\right)^{\eta-2\mu}
\left(\partial_{x_1}h_{2\mu+1}\right)^{\eta-2\mu} \Delta.	
\end{equation}
To see that $\Delta$ is nonzero, we observe that if \[h_{\eta-2\mu+1}=\cdots=h_{2\mu}=0\] then $\Delta$ is the determinant of an upper triangular matrix and its diagonal entries  are all equal to $h_{\eta-2\mu}$ up to a nonzero multiplicative constant. Therefore, by our genericity assumption the determinant \eqref{eq:det} is nonzero. 

The case $\eta=3\mu$ follows in similar fashion. The difference in this case is that the matrix \eqref{eq:fullSylv} has one more column than row; it is of size $(\mu+1)\times(\mu+2)$ and we have to show that it is of rank $\mu+1$. We proceed as previously by selecting the same minors in the blocks $\mathrm{Sylv}_\eta(f_0)$ and $\mathrm{Sylv}_\eta(f_1)$. Then, in the remaining part of the matrix, the minor whose determinant is equal to $\Delta$ is of size $\mu+1$ and hence there is no need to select an element in the last column. The determinant $\Delta$ is nonzero by the same argument and the claimed property \eqref{eq:rankphi} is proved. 

\medskip

Finally, it remains to determine the degree $e$ with respect to $x_0,x_1$ of the minimal syzygy we found, i.e.~the integer $e$ that appears in the exact sequence
$$ 0\rightarrow A(-e)\oplus A(-k)^\mu \rightarrow A(-k+1)^{2\mu}\oplus A(-k)^{2\mu+2} \xrightarrow{\phi_{3\mu}} A^{3\mu+1}.$$
As explained above, the contribution of the bi-Euler syzygy in this sequence corresponds to a well identified block matrix. It follows that the minimal syzygy we are interested is a generator of the kernel of the matrix \eqref{eq:fullSylv}, which corresponds to the following restriction $\overline{\phi}_{3\mu}$ of the map $\phi_{3\mu}$:
$$ 0\rightarrow A(-e) \rightarrow A(-k+1)^{2\mu}\oplus  A(-k)^{\mu+1}\oplus A(-k) \xrightarrow{\overline{\phi}_{3\mu}} A^{3\mu+1}.$$
Therefore, we deduce that
$$A(-e)\simeq \wedge^{3\mu+2}\left(A(-k+1)^{2\mu}\oplus  A(-k)^{\mu+1}\oplus A(-k) \right)
\simeq A(-(3\mu+2)k+2\mu)$$
and hence that
\[
e=(3\mu+2)k-2\mu.
\]
Since $d-k=2\mu+1$, we deduce that $M(f_{(k,d-k)})$ has a minimal first syzygy of bidegree 
\[
\begin{array}{ccc}
(e,3\mu)&=&((3\mu+2)k-2\mu,3\mu)\\ 
              &=&(2k,0)+\mu(3k-2,3)\\
               &=&(2k,0)+\frac{d-k-1}{2}(3k-2,3)
               \end{array}
               \]
as claimed.
\end{proof}

\begin{rk} From the proof of this theorem, we see that the minimal syzygy of $M(f_{k,d-k})$ we obtained is completely explicit: it is a determinantal syzygy that is built from the minors of maximal size of a submatrix of \eqref{eq:fullSylv}.
\end{rk}

\begin{ex}\label{IsTight}
Suppose $d=19$ and $k=6$. Then the minimal bigraded first syzygies are of bidegree (with exponent denoting the rank in that bidegree)
$$\begin{array}{ccccc}
(6, 13), &(8, 68), &(8, 69), &(9, 43)^2, &(9, 44)^2, \\
(10, 26), &(10,37)^5, &(11, 25)^4, &(12, 24), &(16, 24)^{12}, \\
(17, 23)^{10}, &(18, 23), &(19, 22)^8, &(20, 22),& (22, 21)^5, \\
(23, 21)^2, &(28, 20)^5,  &(41, 19)^2, &(42, 19), &(108, 18) 
\end{array}$$
The minimal first syzygy of bidegree $(108,18)$ has degree $126$ in the $\Z$-grading, so 
\[
\reg M(f_{6,13}) \mbox{ is }\ge 124.
\]
\pagebreak

\noindent The betti table (see \cite{Eis2}) for the $\Z$-graded minimal free resolution for $M(f_{6,13})$ is
\begin{verbatim}
            0 1  2   3  4
     total: 1 4 66 107 44
         0: 1 .  .   .  .
          : . .  .   .  .
        17: . 4  1   .  .
          : . .  .   .  .
        34: . .  6   4  1
          : . .  .   .  .
        38: . . 22  33 10
        39: . .  9  18  9
        40: . .  1   2  1
        41: . .  5  10  5
        42: . .  2   4  2
        43: . .  .   .  .
        44: . .  .   .  .
        45: . .  5   6  .
        46: . .  5  10  5
        47: . .  .   .  .
        48: . .  .   .  .
        49: . .  .   .  1
        50: . .  2   4  2
        51: . .  2   4  2
          : . .  .   .  .
        58: . .  2   4  2
        59: . .  1   2  1
          : . .  .   .  .
        74: . .  1   2  1
        75: . .  1   2  1
          : . .  .   .  .
       124: . .  1   2  1
\end{verbatim}
Note that the bound on $d$ is achieved. To save space we have replaced large empty stretches in the table with
an unlabelled row beginning with {\tt :}
\end{ex}

\noindent{\bf Closing Remarks and Questions}:
Theorem~\ref{Codim1SurfReg} illustrates that the additional algebraic
structure of multigraded hypersurfaces provides leverage to obtain
interesting results on the regularity of $M(f)$, and we are working on a sequel to this paper to investigate this question. We thank Carlos D'Andrea for pointing out references \cite{Spo} and \cite{Vasc} to us, and an anonymous referee for helpful comments.
Computations in {\tt Macaulay 2} \cite{EGS} and {\tt Singular}
\cite{DGPS} provided evidence for our work.
\pagebreak


\begin{thebibliography}{20}

\bibitem{Abe} T. Abe, Plus-one generated and next to free arrangements of hyperplanes, 
\textit{Int. Math. Res. Not.}. doi:10.1093/imrn/rnz099.

\bibitem{Beau} A. Beauville, Determinantal hypersurfaces, Michigan Math. J. 48, 39-64 (2000). 

\bibitem{BrHe93} W. Bruns and J. Herzog.
\newblock {\em Cohen-{M}acaulay rings}, volume~39 of {\em Cambridge Studies in Advanced Mathematics}.
\newblock Cambridge University Press, Cambridge, 1993.

\bibitem{Cha} M. Chardin, Applications of some properties of the canonical module in computational
projective algebraic geometry. Symbolic computation in algebra, analysis, and geometry
(Berkeley, CA, 1998). J. Symbolic Comput. 29 (2000),  527--544.

\bibitem{CK} D. Cox, S. Katz, Mirror symmetry and algebraic geometry. Mathematical Surveys and Monographs, 68. American Mathematical Society, (1999).

\bibitem{DGPS} W. Decker, G.-M. Greuel, G. Pfister, H. Sch{\"o}nemann, 
\newblock {\sc Singular} {4-1-2} --- {A} computer algebra system for polynomial computations.
\newblock {http://www.singular.uni-kl.de} (2019).
 \bibitem{DP} A. Dimca, S. Papadima, Hypersurface complements, Milnor fibers and higher homotopy groups of arrangements. Ann. of Math. 158, 473-507 (2003).

\bibitem{DimcaHodge} A. Dimca, On the syzygies and Hodge theory of nodal hypersurfaces. Ann. Univ. Ferrara Sez. VII Sci. Mat. 6, 87-101 (2017). 

\bibitem{DBull}  A. Dimca, Syzygies of Jacobian ideals and defects of linear systems, Bull. Math. Soc. Sci. Math. Roumanie Tome 56(104) No. 2, 2013, 191--203.



\bibitem{DS1} { A. Dimca,  M. Saito}, { Graded  Koszul cohomology and spectrum of certain homogeneous polynomials}, arXiv:1212.1081v3.


\bibitem{DStFScurve} A. Dimca,  G. Sticlaru, Free and nearly free curves vs. rational cuspidal plane curves, Publ. RIMS Kyoto Univ. 54 (2018), 163--179.

\bibitem{DStFS} A. Dimca,  G. Sticlaru, Free and nearly free surfaces in $\mathbb{P}^3$, Asian J. Math. 22 (2018), 787--810.

\bibitem{ELSV} L. Ein, R. Lazarsfeld, K. Smith, D. Varolin,  Jumping  coefficients of multiplier ideals. Duke Math. J. 123, 469-506 (2004). 

\bibitem{Eis1}D. Eisenbud, \newblock {\em Commutative algebra}, volume 150 of {\em Graduate Texts in  Mathematics}.
\newblock Springer-Verlag, New York, 1995.\newblock With a view toward algebraic geometry.

\bibitem{Eis2} { D. Eisenbud}, \emph{The Geometry of Syzygies: A Second Course in Algebraic Geometry and Commutative Algebra},
Graduate Texts in Mathematics, Vol. 229, Springer 2005. 

\bibitem{EGS} D. Eisenbud, D. Grayson, M. Stillman, Macaulay 2: a software system 
    for algebraic geometry and commutative algebra, {\tt 
      http://www.math.uiuc.edu/Macaulay2}
    
\bibitem{Griff} P. Griffiths, On the period of certain rational
  integrals I, II. Ann. Math. 90, 460-541 (1969).
  
\bibitem{GN} T. Gulliksen, O. Neg\.ard, Un complexe r\'esolvant pour certains id\'eaux d\'eterminantiels, C. R. Acad. Sci. Paris, 274, A16-A18 (1972).
 
\bibitem{Joze} T. J\'ozefiak, Ideals generated by minors of a symmetric matrix, Comment. Math. Helv. 53, 595-607 (1978).
 
\bibitem{L} A. Lascoux, Syzygies des vari\'et\'es d\'eterminantales, Adv. in Math. 30, 202-237 (1978).

 \bibitem{MS} M. Mustata, H. Schenck, The module of logarithmic p-forms of a locally free arrangement. J. Algebra 241, 699-719 (2001). 

\bibitem{RT} L. L. Rose, H. Terao,  A free resolution of the module of logarithmic forms of a generic arrangement, J. Algebra 136 (1991), 376--400.

\bibitem{Se} E. Sernesi,  The local cohomology of the Jacobian ring, {Documenta Mathematica},  19 (2014), 541-565. 

\bibitem{Spo} S. Spodzieja, On some property of the Jacobian of a homogeneous polynomial mapping, Bull. Soc. Sci. Lett. Lodz 39, no. 5, 1-6 (1989).

\bibitem{Terao} H. Terao, Generalized exponents of a free arrangement of hyperplanes and Shepherd-Todd-Brieskorn formula. Invent. Math. 63, 159-179 (1981). 

\bibitem{SW} D. van Straten, T. Warmt,  Gorenstein duality for one-dimensional almost complete intersections--with an application to non-isolated real singularities, Math. Proc.Cambridge Phil. Soc.158, 249-268, (2015).

\bibitem{Sz} A. Szanto, with an appendix by Marc Chardin, Solving overdetermined systems by the subresultant method. 
J. Symbolic Computation 43, 46-74 (2008).

\bibitem{Vasc} W.~Vasconcelos, The top of a system of equations.
Boletin de la Sociedad Matematica Mexicana, Vol.~37, 549-556 (1992).

\bibitem{Walther} U. Walther, The Jacobian module, the Milnor fiber, and the D-module generated by $f^s$. Invent. Math. 207, 1239-1287 (2017).

\bibitem{Z} G. Ziegler, Combinatorial construction of logarithmic differential forms, Adv. Math. 76, 116-154 (1989).

\end{thebibliography}
\end{document}